\documentclass{amsart}
\usepackage{amsmath,amsthm,amsfonts,amssymb,amscd,amsbsy,multirow,hyperref}
\usepackage[all]{xy}
\usepackage{tikz}
\usepackage{graphicx,pgf,caption,subcaption}
\usepackage{enumerate}
\usepackage{dsfont}
\usepackage{stmaryrd}
\usepackage{ytableau}
\usepackage{amsmath,amsthm,amsfonts,amssymb,amscd,amsbsy,dsfont}
\usepackage{graphicx}
\usepackage{hyperref}

\newtheorem{corollary}{Corollary}

\newtheorem{lemma}{Lemma}
\newtheorem{proposition}{Proposition}
\newtheorem{remark}{Remark}
\newtheorem{theorem}{Theorem}

\numberwithin{equation}{section}

\title[Critical metrics of the volume functional]{Geometric inequalities for critical metrics\\ of the volume functional}

\author{H. Baltazar}
\author{R. Batista}
\author{E. Ribeiro Jr.}

\address[H. Baltazar]{Universidade Federal do Piau\'{i} - UFPI, Departamento de Matem\'{a}tica, Campus Petr\^onio Portella, 64049-550, Teresina / PI, Brazil} \email{halyson@ufpi.edu.br}

\address[R. Batista]{Universidade Federal do Piau\'i - UFPI, Departamento de Matem\'{a}tica, Campus Petr\^onio Portella, 64049-550, Teresina / PI, Brazil} \email{rmarcolino@ufpi.edu.br}

\address[E. Ribeiro Jr]{Universidade Federal do Cear\'a - UFC, Departamento  de Matem\'atica, Campus do Pici, Av. Humberto Monte, Bloco 914,
60455-760, Fortaleza / CE, Brazil.}\email{ernani@mat.ufc.br}

\thanks{H. Baltazar was partially supported by FAPEPI/Brazil (Grant: 007/2018)}

\thanks{R. Batista was partially supported by CNPq/Brazil (Grant: 310881/2017-0)}

\thanks{E. Ribeiro was partially supported by CNPq/Brazil (Grant: 305410/2018-0 and 160002/2019-2), PRONEX-FUNCAP/CNPq/Brazil, CAPES/ Brazil  - Finance Code 001.}

\date{\today}

\begin{document}

\begin{abstract}
In this article, we investigate the geometry of critical metrics of the volu\-me functional on an $n$-dimensional compact manifold with (possibly disconnected) boundary. We establish sharp estimates to the mean curvature and area of the boun\-da\-ry components of critical metrics of the volu\-me functional on a compact manifold. In addition, localized version estimates to the mean curvature and area of the boundary of critical metrics are also obtained. 
\end{abstract}

\maketitle

\section{Introduction}
\label{intro}
A promising way to find canonical metrics on a given manifold is to investigate critical metrics which arise as solutions of the Euler-Lagrange equations for curvature functionals.  It is known that the critical points of the total scalar curvature functional restricted to the set of smooth Riemannian structures on a compact manifold $M^n$ of unitary volume must be ne\-cessarily Einstein (see \cite[Theorem 4.21]{besse}). In a similar context, Miao and Tam \cite{miaotam,miaotamTAMS} and Corvino, Eichmair and Miao \cite{CEM} investigated the modified problem of finding stationary points for the volume functional on the space of metrics whose scalar curvature is equal to a given constant. 

In order to make our approach more understandable, we need to fix some terminology (see \cite{BDR,BDRR}).  Let $(M^{n},\,g)$ be a connected compact Riemannian manifold with boundary $\partial M.$ We say that $g$ is, for brevity, a {\it Miao-Tam critical metric} (or simply, {\it critical metric}), if there is a nonnegative smooth function $f$ on $M^n$ such that $f^{-1}(0)=\partial M$ satisfying the overdetermined-elliptic system
\begin{equation}\label{eqMiaoTam1}
\mathfrak{L}_{g}^{*}(f)=-(\Delta f)g+Hess_{g} f-fRic_{g}=g.
\end{equation}  Here, $\mathfrak{L}_{g}^{*}$ is the formal $L^{2}$-adjoint of the linearization of the scalar curvature operator $\mathfrak{L}_{g}$. Moreover, $Ric,$ $\Delta$ and $Hess$ stand for the Ricci tensor, the Laplacian operator and the Hessian form on $M^n,$ respectively.

Miao and Tam \cite{miaotam} showed that these critical metrics arise as critical points of the volume functional on $M^n$ when restricted to the class of metrics $g$ with pres\-cribed constant scalar curvature such that $g_{|_{T \partial M}}=h$ for a pres\-cribed Riemannian metric $h$  on the boundary; see also \cite{CEM}. Such metrics are effectively relevant in understanding the influence of the scalar curvature in controlling the volume of a given manifold. In this context, Corvino, Eichmair and Miao \cite{CEM} were able to establish a deformation result which suggests that the information of scalar curvature is not sufficient in giving volume comparison. This is in fact important because the volume comparison results are often used to explore geometrical and topological properties of a given manifold. Explicit examples of critical me\-trics can be found in  \cite{miaotam,miaotamTAMS}. They include the spatial Schwarzschild metrics and AdS-Schwarzschild metrics restricted to certain domains containing their horizon and bounded by two spherically symmetric spheres. Besides the standard metrics on geodesic balls in space forms are critical metrics. There are several uniqueness results in the literature. For more details on this subject see, e.g.,  \cite{BalRi1,BDR21,BDR,BS,BDRR,CEM,FangYuan,Kim,miaotam,miaotamTAMS,SW,yuan}.

 In \cite{CEM}, Corvino, Eichmair and Miao showed that the area of the boundary $\partial M$ of an $n$-dimensional scalar flat Miao-tam cri\-ti\-cal me\-tric must have an upper bound depending on the volume of $M^n$ (see \cite[Proposition 2.5]{CEM}). Later, inspired by a classical result obtained by Boucher,  Gibbons and Horowitz \cite{BGH} and Shen \cite{Shen}, Batista et al. \cite{BDRR} showed that the boundary $\partial M$ of a compact three-dimensional oriented Miao-Tam critical metric $(M^3,\,g)$ with connected boundary and nonnegative scalar curvature must be a $2$-sphere whose area satisfies the inequality $area(\partial M)\leq\frac{4\pi}{C},$ where $C$ is a constant greater than $1.$ This result also holds for negative scalar curvature, provided that the mean curvature of the boundary satisfies $H > 2,$ as was proved in \cite{BLF} and \cite{BS}.  In \cite{BalRi2}, Baltazar, Di\'ogenes and Ribeiro obtained an isoperimetric type inequality for Miao-Tam critical metrics with nonnegative scalar curvature.  Indeed, in the last years several progress have been made on boundary and volume estimates for critical metrics (see, e.g., \cite{BS,BLF,BDRR,CEM,yuan}), however,  only few results are known for possibly disconnected boundary case, which is the focus of this paper.  These estimates can be used to obtain new classification results and to discard some possible new examples.

In our first result, mainly inspired by the works of Borghini and Mazzieri \cite{BM1,BM2} on the uniqueness result for the de Sitter solution, we shall provide a sharp estimate to the mean curvature $H_{i}$ of the boundary components $\partial M_{i}$ of a critical metric of the volume functional on an $n$-dimensional compact manifold. More precisely, we have established the following result.

\begin{theorem}\label{thmMainA}
Let $(M^n,\,g,\,f)$ be an $n$-dimensional compact Miao-Tam critical me\-tric with (possibly disconnected) boundary $\partial M.$ Then we have:
\begin{eqnarray}
\min H_{i}\leq\sqrt{\frac{n(n-1)}{R(f_{\max})^{2}+2nf_{\max}}}\,\,\,\,\,\hbox{on}\,\,\,\,\partial M,
\end{eqnarray} where $f_{\max}$ is the maximum value of $f.$ Moreover, equality holds if and only if $M^{n}$ is isometric to a geodesic ball in a simply connected space form $\mathbb{R}^{n},$ $\mathbb{S}^{n}$ or  $\mathbb{H}^{n}.$
\end{theorem}

\begin{remark}
A relevant  observation is that the mean curvature of the boundary of geodesic balls in space form satisfies $H^{2}=\frac{n(n-1)}{R(f_{\max})^{2}+2nf_{\max}}.$ Therefore, it follows from Theorem \ref{thmMainA} that the mean curvature of the boundary of a geodesic ball is the ma\-xi\-mum possible among all Miao-Tam critical metrics on compact manifolds with connected boundary.  Moreover, no restriction on the sign of the scalar curvature is assumed.
\end{remark}

A key ingredient to establish the proof of Theorem \ref{thmMainA} is a Robinson-Shen type identity (see Lemma \ref{Shen} in Section \ref{Sec2}) which is essentially inspired by \cite{Rob,Shen}.  Recently, a similar idea was used by Di\'ogenes, Gadelha and Ribeiro \cite{DGR} to obtain a sharp geometric inequality for compact quasi-Einstein manifolds with boundary.

As a consequence of Theorem \ref{thmMainA} we obtain the following corollary for scalar flat critical metrics.

\begin{corollary}
\label{corthm1}
Let $(M^n,\,g,\,f)$ be an $n$-dimensional compact scalar flat Miao-Tam critical me\-tric with connected boundary $\partial M.$ Then we have:

\begin{equation}
area(\partial M)\leq \sqrt{\frac{n^{2}}{2(n-1)f_{\max}}} Vol(M).
\end{equation} Moreover,  equality holds if and only if $M^{n}$ is isometric to a geodesic ball in $\mathbb{R}^{n}$ of radius $\sqrt{2(n-1)f_{\max}}.$
\end{corollary}

It is well-known that if $\Omega$ is a bounded domain with smooth boundary in a simply connected space form, then the corresponding space form metric is a Miao-Tam critical metric on $\Omega$ if and only if $\Omega$ is a geodesic ball (see \cite{miaotamTAMS}). Therefore, it is natural to ask whether geodesic balls are the only critical metrics whose boundary is isometric to a standard sphere (see [\cite{miaotamTAMS}, pg. 2908]). As an application of Theorem \ref{thmMainA}, we obtain the following partial answer to this question. 

\begin{corollary}
\label{cor2thm1}
Let $(M^n,\,g,\,f)$ be an $n$-dimensional compact scalar flat Miao-Tam critical metric with boundary isometric to a standard sphere $\mathbb{S}^{n-1}(r)$. If $r>\sqrt{2(n-1)f_{max}}$, then $(M^n,\,g)$ can not be isometric to a geodesic ball in a simply connected space form $\mathbb{R}^n$.
\end{corollary}

Proceeding, we shall establish a localized version of Theorem \ref{thmMainA} which can be also seen as an obstruction result.  In order to do so, we consider $MAX(f)$ to be the set where the maximum of $f$ is achieved, namely, $$MAX(f)=\{p\in M;\,f(p)=f_{max}\}$$ and let $E$ be a single connected component of $M\setminus MAX(f).$ 

Now we may state our next result.

\begin{theorem}
\label{thmlocalA}
Let $(M^n,\,g,\,f)$ be an $n$-dimensional compact Miao-Tam critical me\-tric with (possibly disconnected) boundary $\partial M,$ let $E$ be a single connected component of $M\setminus MAX(f),$ and let $\Gamma_{E}=\partial M\cap E$ be non-empty and possibly disconnected. Then we have:
\begin{eqnarray}
\min H_{i}\leq\sqrt{\frac{n(n-1)}{R(f_{\max})^{2}+2nf_{\max}}}\ \ \ \ on\ \ \Gamma_{E}.
\end{eqnarray} Moreover, equality holds if and only if $M^{n}$ is isometric to a geodesic ball in a simply connected space form $\mathbb{R}^{n},$ $\mathbb{S}^{n}$ or  $\mathbb{H}^{n}.$
\end{theorem}

\begin{remark}
Notice that Theorem \ref{thmlocalA} guarantees that if on a single connected component $E$ of $M\setminus MAX(f)$ it holds $$H^{2}\ge \frac{n(n-1)}{R(f_{\max})^{2}+2nf_{\max}}\ \ \ \ on\ \ \Gamma_{E},$$ then the entire manifold must be isometric to a geodesic ball in $\Bbb{R}^n,$ $\Bbb{H}^{n}$ or $\Bbb{S}^n.$
\end{remark}

Before presenting our next result, it is important to recall an useful estimate to the area of the boundary $\partial M$ of an $n$-dimensional Miao-tam cri\-ti\-cal me\-tric which was obtained by contributions by \cite{BLF,BDRR,CEM}. 

\begin{theorem}[\cite{BLF,BDRR} and \cite{CEM}]
\label{thmboundary}
Let $(M^n,\,g,\,f)$ be an $n$-dimensional compact Miao-Tam critical metric with connected boundary $\partial M.$ For the case of negative scalar curvature we assume in addition that $H^{2}>-\frac{n-1}{n}R.$ Then the area of the boundary $\partial M$ satisfies

\begin{equation}
\label{eqthmbound}
area(\partial M)\leq \frac{1}{\Big(\frac{n-2}{n}R+\frac{n-2}{n-1}H^{2}\Big)} \int_{\partial M}R^{\partial M}dS,
\end{equation} where $R^{\partial M}$ stands for the scalar curvature of $(\partial M,g_{|_{\partial M}}).$ Moreover, equality holds if and only if $M^{n}$ is isometric to a geodesic ball in $\Bbb{R}^n,$ $\Bbb{H}^{n}$ or $\Bbb{S}^n.$
\end{theorem}

For the sake of completeness and comparison, we will include in Section \ref{Sec2} an alternative proof of Theorem \ref{thmboundary} by using the Robinson-Shen type identity obtained in Lemma \ref{Shen}. A crucial advantage is that the rearrangement of such a proof in terms of the Robinson-Shen identity yields a much more simple presentation.

Our next result is a localized version of Theorem \ref{thmboundary}.

\begin{theorem}
\label{thmB}
Let $(M^n,\,g,\,f)$ be an $n$-dimensional compact Miao-Tam critical me\-tric with (possibly disconnected) boundary $\partial M,$ let $E$ be a single connected component of $M\setminus MAX(f),$ and let $\Gamma_{E}=\partial M\cap E$ be the non-empty and possibly disconnected. Then we have:

\begin{equation}
\label{eqthmB}
\int_{\Gamma_{E}}|\nabla f|\Big(R^{\Gamma_{E}}-\frac{n-2}{n}R-\frac{n-2}{n-1}H^{2}\Big)dS \ge 0,
\end{equation} where $R^{\Gamma_{E}}$ is the scalar curvature of $(\Gamma_{E},\,g_{|_{\Gamma_{E}}}).$ Moreover, equality holds in (\ref{eqthmB}) if and only if $M^{n}$ is isometric to a geodesic ball in $\Bbb{R}^n,$ $\Bbb{H}^{n}$ or $\Bbb{S}^n.$
\end{theorem} 

As an immediate consequence of Theorem \ref{thmB} we get the following corollary.

\begin{corollary}
Let $(M^3,\,g,\,f)$ be a three-dimensional compact Miao-Tam critical metric with (possibly disconnected) boundary. Let $E$ be a single connected component of $M\setminus MAX(f)$ and suppose that  $\Gamma_{E}=\partial M\cap E=\displaystyle{\cup_{i=1}^{k}\Gamma_{E_{i}}}.$ Then we have:

$$\sum_{i=1}^{k}\frac{1}{H_{i}}\Big(\frac{R}{6}+\frac{H_{i}^{2}}{4}\Big)area(\Gamma_{E_i})\leq 2\pi \sum_{i=1}^{k}\frac{\chi(\Gamma_{E_{i}})}{H_{i}}.$$ Moreover, if the equality holds, then $M^3$ must be a geodesic ball in $\Bbb{R}^3,$ $\Bbb{H}^{3}$ or $\Bbb{S}^3.$

\end{corollary}

Supposing that $\Gamma_{E}$ is connected, we obtain the following result which can be compared with \cite[Theorem 2]{BDRR}.

\begin{corollary}
Let $(M^3,\,g,\,f)$ be a three-dimensional compact Miao-Tam critical metric with boundary. Let $E$ be a single component of $M\setminus MAX(f)$ and suppose that $\Gamma_{E}=\partial M\cap E$ is connected. In addition, for the case of negative scalar we assume that $H^{2}>-\frac{2}{3}R.$ Then $\Gamma_{E}$ is homeomorphic to a $2$-sphere and $$area(\Gamma_{E})\leq \frac{4\pi}{\Big(\frac{R}{6}+\frac{H^{2}}{4}\Big)}.$$ Moreover, equality holds if and only if $M^{3}$ is isometric to a geodesic ball in $\Bbb{R}^3,$ $\Bbb{H}^{3}$ or $\Bbb{S}^3.$
\end{corollary}

\section{Background}
\label{Sec2}

In this section, we will review some basic facts and key lemmas that will be useful for the establishment of the main results. To begin with, we remember that the fundamental equation of a Miao-Tam critical metric, in the tensorial notation, is given by
\begin{equation}
\label{eqfund1} -(\Delta f)g_{ij}+\nabla_{i}\nabla_{j} f-fR_{ij}=g_{ij}.
\end{equation} Taking the trace of (\ref{eqfund1}) we arrive at
\begin{equation}
\label{eqtrace} \Delta f +\frac{R}{n-1}f=-\frac{n}{n-1}.
\end{equation} Putting these facts together, we get
\begin{equation}
\label{eqVstaic2}\nabla_{i}\nabla_{j}f-fR_{ij}=-\frac{Rf+1}{n-1}g_{ij}.
\end{equation} Also, it easy to check from (\ref{eqtrace}) that
\begin{equation}
\label{IdRicHess} f\mathring{Ric}=\mathring{Hess f},
\end{equation}  where  $\mathring{T}$ stands for the traceless of $T.$ Furthermore, a Riemannian manifold $(M^{n},\,g)$ for which there exists a nontrivial function $f$ satisfying (\ref{eqfund1}) must have constant scalar curvature $R$ (see \cite[Proposition 2.1]{CEM} and \cite[Theorem 7]{miaotam}).

It is also important to recall that, choosing appropriate coordinates, $f$ and $g$ are analytic. Hence, the set of regular points of $f$ is dense in $M^n$ (see \cite{CEM}, Proposition 2.1). Thus, at regular points of $f,$ the vector field $\nu=-\frac{\nabla f}{|\nabla f|}$ is normal to $\partial M$ and $|\nabla f|$ is constant (non null) on each connected component of $\partial M$ (see \cite{BDRR}, Sec. 3).

The following lemma, obtained previously in [\cite{BDRR}, Lemma 5], will be useful.

\begin{lemma}
\label{lemB1}
Let $(M^{n},g,f)$ be an $n$-dimensional compact Miao-Tam critical metric with smooth boundary $\partial M.$ Then we have:
\begin{equation}
\label{eqLb}
 \int_{M}f|\mathring{Ric}|^{2}dM_{g}=-\frac{1}{|\nabla f|_{\mid_{\partial M}}}\int_{\partial M}\mathring{Ric}(\nabla f,\nabla f)dS.
\end{equation}
\end{lemma}

For our purpose,  we need to provide a Robinson-Shen type identity which is essentially motivated by \cite{Rob,Shen} and plays a crucial role in the paper.

\begin{lemma}\label{Shen} (Robinson-Shen Type Identity)
Let $(M^{n},\,g)$ be a Riemannian manifold satisfying (\ref{eqfund1}). Then we have:
\begin{eqnarray*}
{\rm div}\left[\frac{1}{f}\nabla|\nabla f|^{2}-\frac{2\Delta f}{nf}\nabla f\right]=\frac{2}{f}|\mathring{Hessf}|^{2}
\end{eqnarray*} for $\{x\in M\,:\,f(x)>0\},$ where $\mathring{Hess\,f}=Hess\,f-\frac{\Delta f}{n}g.$
\end{lemma}
\begin{proof}
First of all, we set $X=\frac{1}{f}\nabla|\nabla f|^{2}-\frac{2\Delta f}{nf}\nabla f$ as a vector field in the interior of $M^n$ in order to deduce
\begin{eqnarray*}
\frac{f}{2}{\rm div}\,X&=&\frac{1}{2}\Delta|\nabla f|^{2}+\frac{f}{2}\langle\nabla f^{-1},\nabla|\nabla f|^{2}\rangle-\frac{1}{n}(\Delta f)^{2}-\frac{f}{n}\left\langle\nabla\left(\frac{\Delta f}{f}\right),\nabla f\right\rangle\\
&=&\frac{1}{2}\Delta|\nabla f|^{2}-\frac{1}{f}Hessf(\nabla f,\nabla f)-\frac{1}{n}(\Delta f)^{2}\nonumber\\&&-\frac{1}{n}\langle\nabla\Delta f,\nabla f\rangle+\frac{\Delta f}{nf}|\nabla f|^{2},
\end{eqnarray*} which can be rewritten, using the classical Bochner's formula, as follows
\begin{eqnarray*}
\frac{f}{2}{\rm div}X&=&|\mathring{Hessf}|^{2}+Ric(\nabla f,\nabla f)-\frac{1}{f}Hessf(\nabla f,\nabla f)\\
&&+\frac{n-1}{n}\langle\nabla\Delta f,\nabla f\rangle+\frac{\Delta f}{nf}|\nabla f|^{2}.
\end{eqnarray*} Hence, it suffices to use Eqs. (\ref{eqfund1}) and (\ref{eqtrace}) to obtain
\begin{eqnarray*}
\frac{f}{2}{\rm div}X&=&|\mathring{Hessf}|^{2}-\frac{(\Delta f+1)}{f}|\nabla f|^{2}-\frac{R}{n}|\nabla f|^{2}+\frac{\Delta f}{nf}|\nabla f|^{2}\\
&=&|\mathring{Hessf}|^{2}.
\end{eqnarray*} This completes the proof of the lemma.
\end{proof}

As anticipated, for the reader’s convenience, we include here an alternative proof of Theorem \ref{thmboundary} as a consequence of the Robinson-Shen type identity obtained in Lemma \ref{Shen}.

\subsection{Proof of Theorem \ref{thmboundary}}
\begin{proof} Initially, we invoke Lemma \ref{Shen} jointly with (\ref{IdRicHess}) to infer
\begin{eqnarray}
\label{eq3dr}
{\rm div}\left[\frac{1}{f}\nabla|\nabla f|^{2}-\frac{2\Delta f}{nf}\nabla f\right]&=&\frac{2}{f}|\mathring{Hessf}|^{2}\nonumber\\&=&2f|\mathring{Ric}|^{2}.
\end{eqnarray}

On the other hand, by the fundamental equation we immediately have

\begin{equation}
\label{eq34dg}
\frac{1}{f}\Big(\nabla |\nabla f|^{2}-\frac{2\Delta f}{n}\nabla f\Big)=2\Big(Ric(\nabla f)-\frac{R}{n}\nabla f\Big).
\end{equation} This substituted into (\ref{eq3dr}) gives 
\begin{equation}
\label{eq12s}
{\rm div} \Big(Ric(\nabla f)-\frac{R}{n}\nabla f\Big)=f|\mathring{Ric}|^{2}
\end{equation} and hence, by using the divergence lemma we get

\begin{eqnarray}
\label{eq12s1}
\int_{\partial M}\langle Ric(\nabla f)-\frac{R}{n}\nabla f, \nu\rangle dS\geq 0,
\end{eqnarray} so that
\begin{equation}
\label{eqka}
\int_{\partial M}Ric(\nu,\nu)dS\le \frac{R}{n}area(\partial M).
\end{equation}

At the same time, the second fundamental form of $\partial M$  is given by
\begin{equation}
\label{secff}
h_{ij}=\langle \nabla_{e_{i}}\nu,e_{j}\rangle,
\end{equation} where $\{e_{1},\ldots,e_{n-1}\}$ is  an orthonormal frame on $\partial M.$ Then,  one sees that

\begin{equation}
\label{1a}
h_{ij}=-\langle \nabla_{e_{i}}\frac{\nabla f}{|\nabla f|},e_{j}\rangle=\frac{1}{(n-1)|\nabla f|}g_{ij}.
\end{equation} Moreover, the mean curvature $H=\frac{1}{|\nabla f|}$  is constant and $\partial M$ is totally umbilical (see also \cite[Theorem 7]{miaotam}). The Gauss equation implies

\begin{equation}
\label{2a1}
2Ric(\nu,\nu)+R^{\partial M}=R+\frac{(n-2)}{(n-1)}H^{2},
\end{equation} where $R^{\partial M}$ stands for the scalar curvature of $(\partial M,g_{|_{\partial M}});$ for more details see  Eq. (45) in \cite{CEM}. Now, comparing (\ref{2a1}) with (\ref{eqka}) we obtain

\begin{equation}
area(\partial M)\leq \frac{1}{\Big(\frac{n-2}{n}R+\frac{n-2}{n-1}H^{2}\Big)} \int_{\partial M}R^{\partial M}dS,
\end{equation} which gives the asserted inequality. 

Next, supposing that the equality holds, it suffices to return to Eqs. (\ref{eq12s}) and (\ref{eq12s1})  in order to conclude that $\int_{M}f|\mathring{Ric}|^{2}dM_{g}=0.$ But, since $f$ cannot vanish identically in a non-empty open set, we deduce that $|\mathring{Ric}|^{2}=0$ and therefore, $M^{n}$ is an Einstein manifold. Then, we invoke Theorem 1.1 in \cite{miaotamTAMS} to conclude that $M^{n}$ is isometric to a geodesic ball in $\Bbb{R}^n,$ $\Bbb{H}^{n}$ or $\Bbb{S}^n.$ So, the proof is completed.
\end{proof}

In order to proceed, it is crucial to recall a gradient estimate, so called {\it the reverse Łojasiewicz inequality}, on the behaviour of an analytic function near a critical point (see \cite[Theorem 1.1.7]{tesestefano}).  

\begin{theorem}[Reverse Łojasiewicz Inequality, \cite{tesestefano}]
\label{reverse_Lojasiewicz}
Let $(M^{n},\,g)$ be a Riemannian mani\-fold, let $f:M\to \Bbb{R}$ be a smooth function and $p\in M$ be a local maximum point satisfying:

\begin{enumerate}
\item $|\nabla^{2}f|(p)\neq 0,$
\item the set $\{x\in M;\,f(x)=f(p)\}$ is compact. Here, every point in the set $\{f(x)=f(p)\}$ is a local maximum.
\end{enumerate} Then for every $\theta<1,$ there exists a neighborhood  $U_{p}$ containing $p$ and a real number $C_{p}>0$ such that for every $x\in U_{p}$ it holds
$$|\nabla f|^{2}(x)\leq C_{p}\left(f(p)-f(x)\right)^{\theta}.$$
\end{theorem}

\begin{remark}
\label{remA}
We highlight that the condition $|\nabla^{2}f|(p)\neq 0$ in Theorem  \ref{reverse_Lojasiewicz} was recently removed by  Borghini, Chru\'sciel and Mazzieri in \cite[Theorem 2.2]{BCM}. In other words, the Reverse Lojasiewicz Inequality remains true without any assumption on the Hessian.
\end{remark}

It is known that gradient estimates are important and essential for deriving convergence results in different geometric flows, and most of them are obtained by analytic methods. Roughly speaking, the reverse Łojasiewicz inequality tells us that the bound for the gradient near critical points can be obtained around the local maxima (or local minima) under suitable conditions. This estimate will play a crucial role in the proof of Theorem \ref{thmB}. For more details and other applications see, for instance, \cite{FM}.

Now, following the procedure adopted in \cite{BM1}, we shall present a No Islands Lemma which is a slightly modification of Lemma 5.1 in \cite{BM1}. To be precise, we shall show that the set $\Gamma_{E}=\partial M\cap E$ is always nonempty, where $E$ is a connected component of $M\setminus MAX(f).$

\begin{proposition}
\label{prop1a}
Let $(M^{n},\,g,\,f)$  be an $n$-dimensional compact Miao-Tam critical metric with boundary $\partial M$ (possibly disconnected). Let $E$ be a connected component of $M\setminus MAX(f).$ Then $\Gamma_{E}= \partial M\cap E \neq \emptyset.$
\end{proposition}

\begin{proof}
We argue by contradiction, assuming that $\Gamma_{E}= \emptyset.$ In this situation, since $MAX(f)\cap \partial M=\emptyset$ we deduce that 

\begin{equation}
\label{eq54g1}
\overline{E}\setminus E\subset MAX(f),
\end{equation} where $\overline{E}$ stands for the closure of $E$ in $M.$ 

From now on we divide the proof in two cases. Firstly, we assume that $M^n$ has nonnegative scalar curvature. Whence, we have from (\ref{eqtrace}) that $\Delta f\leq 0$ in $E.$ Then, we may apply the Weak Maximum Principle and (\ref{eq54g1}) to infer

\begin{equation}
\min_{\overline{E}} f=\min_{\overline{E}\setminus E}f\geq \min_{MAX(f)}f=f_{max}.
\end{equation} Hence,  $f\equiv f_{max}$ in the open set $E,$ and by analyticity of $f$ jointly with
fact that $f|_{\partial M} = 0$ we immediately have that $f\equiv 0$ on $M^{n}.$ This leads to a contradiction.

On the other hand, since $M^n$ has negative scalar curvature, it follows from (\ref{eqtrace}) and the fact that the laplacian of $f$ is nonpositive at the maximum points, one sees that $f(q)\le -\frac{n}{R}$ for any point $q\in MAX(f)$ and hence, one sees that

\begin{eqnarray}
\Delta f&\leq& \frac{R}{n-1}\frac{n}{R}-\frac{n}{n-1}=0\,\,\,\,\,\,\hbox{on}\,\,\,\,E.
\end{eqnarray} In order to obtain a contradiction it therefore suffices to repeat the same arguments used in the proof of the first part. This finishes the proof of the proposition.
\end{proof}

To proceed, it is essential to highlight that on a Miao-Tam critical metric there exists a function  $\Phi$ satisfying

\begin{equation}\label{EqPHI0}
\Delta\Phi-\frac{1}{f}\langle \nabla f,\nabla\Phi\rangle=2|\mathring{Hessf}|^{2}\geq 0.
\end{equation} Indeed, for $R=0,$ it suffices to combine Eq.  (\ref{eqtrace}) and Lemma ~\ref{Shen} to obtain 
\begin{equation}
\label{eqPhiR10}
\Phi=|\nabla f|^{2}+\frac{2}{n-1}f.
\end{equation} Otherwise, if $R\neq 0,$ we can use (\ref{eqfund1}) jointly with Lemma~\ref{Shen} to deduce 

\begin{eqnarray*}
2|\mathring{Hessf}|^{2}&=&f{\rm div}\left[\frac{1}{f}\nabla|\nabla f|^{2}+\frac{2(n-1)\Delta f}{nRf}\nabla\Delta f \right]\\
&=&f{\rm div}\left[\frac{1}{f}\nabla|\nabla f|^{2}+\frac{(n-1)}{nRf}\nabla(\Delta f)^{2} \right]\\
&=&f{\rm div}\left[\frac{1}{f}\nabla\left(|\nabla f|^{2}+\frac{(n-1)}{nR}(\Delta f)^{2}\right) \right].
\end{eqnarray*} Whence, we get
 \begin{equation}
 \label{eqPhiR20}
\Phi=|\nabla f|^{2}+\frac{(n-1)}{nR}(\Delta f)^{2}.
\end{equation}

Now we must pay attention to the coefficient $1/f$ which appears in (\ref{EqPHI0}). The next lemma will be useful to deal with this obstacle. 

\begin{lemma}
\label{lemK1}
Let $(M^{n},\,g,\,f)$  be an $n$-dimensional compact Miao-Tam critical metric with boundary $\partial M$ (possibly disconnected). Let $E$ be a connected component of $M\setminus MAX(f)$ with $\Gamma_E=\partial M\cap E$ nonempty (possibly disconnected). Suppose that $$H^{2}\geq \frac{n(n-1)}{Rf_{max}^{2}+2nf_{max}}\,\,\,\,\,\,\hbox{on}\,\,\,\Gamma_E.$$ Then we have

$$|\nabla f|^{2}\leq \frac{R(f_{max}^{2}-f^{2})+2n(f_{max}-f)}{n(n-1)}$$ on the whole $E.$

\end{lemma}  
\begin{proof}  We start the proof recalling that, as a consequence of Lemma~\ref{Shen}, we get

$$\Delta\Phi-\frac{1}{f}\langle \nabla f,\nabla\Phi\rangle=2|\mathring{Hessf}|^{2},$$ where the function $\Phi$ is given by
$$\Phi=\left\{%
\begin{array}{ll}
    \displaystyle |\nabla f|^{2}+\frac{2}{(n-1)}f, & \hbox{if $R=0,$} \\
    \\
    \displaystyle |\nabla f|^{2}+\frac{(n-1)}{nR}(\Delta f)^{2}, & \hbox{if $R\neq0.$} \\
\end{array}%
\right.$$

Now, we consider subdomains $E_{\varepsilon}=E\cap \{\varepsilon \leq f\leq f_{max}-\varepsilon\}$ for $\varepsilon$ sufficiently small. Moreover, we already know by analyticity of $f$ that the set of critical values of $f$ is discrete. Hence, there exists a $\delta>0$ such that for every $0<\varepsilon\leq \delta$ the level sets $\{f=\varepsilon\}$ and $\{f=f_{max}-\varepsilon\}$ are regular. Notice also that $\frac{1}{f}\leq \frac{1}{\varepsilon}$ on $E_{\varepsilon}.$  

Proceeding, we apply the Maximum Principle for the function $\Phi$ in order to infer

\begin{equation}\label{MPphi}
\max_{E_{\varepsilon}}\Phi\leq\max_{\partial E_{\varepsilon}}\Phi.
\end{equation} 

On the other hand, it is easy to check that

$$\Phi|_{MAX(f)}=\left\{%
\begin{array}{ll}
    \displaystyle\frac{2}{n-1}f_{\max}, & \hbox{if $R=0,$} \\
    \\
    \displaystyle\frac{n-1}{nR}\left(\frac{Rf_{\max}+n}{n-1}\right)^{2}, & \hbox{if $R\neq0,$} \\
\end{array}%
\right.$$

and

$$\Phi|_{\partial M}=\left\{%
\begin{array}{ll}
    \displaystyle|\nabla f|^{2}_{\mid_{\partial M}}, & \hbox{if $R=0,$} \\
    \\
    \displaystyle|\nabla f|^{2}_{\mid_{\partial M}}+\frac{n}{(n-1)R}, & \hbox{if $R\neq0.$} \\
\end{array}%
\right.$$ Consequently, taking into account that $H=\frac{1}{|\nabla f|}$ on $\partial M,$ our assumption implies

 $$\Phi|_{\Gamma_E}\leq\left\{%
\begin{array}{ll}
    \displaystyle \frac{2}{n-1}f_{\max},  & \hbox{if $R=0,$} \\
    \\
    \displaystyle \frac{n-1}{nR}\left(\frac{Rf_{\max}+n}{n-1}\right)^{2}, & \hbox{if $R\neq0.$} \\
\end{array}%
\right.$$ Finally, letting $\varepsilon\to 0^{+}$ in (\ref{MPphi}) we achieve at

 $$\Phi\leq\left\{%
\begin{array}{ll}
    \displaystyle \frac{2}{n-1}f_{\max},  & \hbox{if $R=0,$} \\
    \displaystyle \frac{n-1}{nR}\left(\frac{Rf_{\max}+n}{n-1}\right)^{2}, & \hbox{if $R\neq0,$} \\
\end{array}%
\right.$$ on the whole $E.$ This finishes the proof of the lemma. 
\end{proof}

\section{Proof of the Main Results}

\subsection{Proof of Theorem \ref{thmMainA}}

\begin{proof}
To begin with, taking into account that the mean curvature on each component of the boundary is given by $H_{i}=\frac{1}{|\nabla f|}_{\big|\partial M_{i}},$ where $\partial M_{i}$ is the component $i$ of the boundary of $M^n,$ it suffices to show that if a Miao-Tam critical metric satisfies
\begin{eqnarray}\label{inequality}
|\nabla f|^{2}_{|\partial M}\leq \frac{Rf_{\max}^{2}+2nf_{\max}}{n(n-1)},
\end{eqnarray} then there is only one possibility for such a manifold, that is, $M^{n}$ must be isometric to a geodesic ball in $\mathbb{R}^{n},$ $\mathbb{S}^{n}$ or $\mathbb{H}^{n}.$ 

In order to do so, let us first recall that, choosing appropriate coordinates, $f$ and $g$ are analytic (see  \cite[Proposition 2.1]{CEM}). Hence, $f$ cannot vanish identically in a non-empty open set. In particular, we can choose a positive number $\eta>0$ such that,  for $0<\varepsilon\leq\eta,$ each level set $\{x\in M;\;f(x)=\varepsilon\}$ is regular. 

From now on we consider, for simplicity, the set $M_{\varepsilon}=\{f\geq\varepsilon\}.$ In this case, we use the Maximum Principle into (\ref{EqPHI0}) to infer
$$\max_{M_{\varepsilon}}\Phi\leq\max_{\partial M_{\varepsilon}}\Phi.$$ Therefore, for $R=0,$ we have from (\ref{eqPhiR10}) that

$$\displaystyle\lim_{\varepsilon\rightarrow0^{+}}\max_{\partial M_{\varepsilon}}\Phi=|\nabla f|^{2}_{|\partial M},$$ and similarly, for $R\neq0,$ we obtain by (\ref{eqPhiR20}) that
$$\displaystyle\lim_{\varepsilon\rightarrow0^{+}}\max_{\partial M_{\varepsilon}}\Phi=|\nabla f|^{2}_{|\partial M}+\frac{n}{(n-1)R}.$$ Therefore, it suffices to use (\ref{inequality}) to arrive at

$$\max_{\partial M_{\eta}}\Phi\leq\frac{2}{n-1}f_{\max},\;\;\;\;\hbox{when}\,\,R=0$$
and
$$\max_{\partial M_{\eta}}\Phi\leq\frac{n-1}{nR}\left(\frac{Rf_{\max}}{n-1}+\frac{n}{n-1}\right)^{2},\;\;\;\;\hbox{when}\,\,R\neq 0.$$

To proceed notice that the set $\{p\in M; f(p)=f_{\max}\}\subset M_{\eta},$ and for every $p\in\{f(p)=f_{\max}\}$ we have 
$$\Phi(p)=\frac{2f_{\max}}{n-1},\;\;\;\;\hbox{if}\,\,R=0$$
and
$$\Phi(p)=\frac{n-1}{nR}\left(\frac{Rf_{\max}}{n-1}+\frac{n}{n-1}\right)^{2},\;\;\;\;\hbox{if}\,\,R\neq 0.$$

Finally, it follows from the Strong Maximum Principle that $\Phi$ is identically cons\-tant on $M_{\eta},$ and since $\eta$ was chosen arbitrarily small, we infer that $\Phi$ is cons\-tant on $M.$ Therefore, returning to Eq. (\ref{EqPHI0}) we use (\ref{IdRicHess}) to conclude that $M$ is an Einstein manifold. Now, we are in position to use Theorem 1.1  of \cite{miaotamTAMS} to conclude that $M^{n}$ is isometric to a geodesic ball in a simply connected space form $\mathbb{R}^{n}$ or $\mathbb{S}^{n}$ or  $\mathbb{H}^{n}.$ This finishes the proof of the theorem.

\end{proof}

\subsection{Proof of Corollary \ref{corthm1}}
\begin{proof}
On integrating (\ref{eqtrace}) we get

$$\int_{M}\Delta f\, dM_{g}=-\frac{n}{n-1}Vol(M),$$ and using the Stokes' formula we obtain

\begin{equation*}
\frac{area(\partial M)}{H}=\frac{n}{n-1}Vol(M).
\end{equation*} In conjunction with Theorem \ref{thmMainA}, one sees that

\begin{equation}
\frac{area(\partial M)}{\sqrt{\frac{n(n-1)}{R(f_{\max})^{2}+2nf_{\max}}}}\leq \frac{n}{n-1}Vol(M).
\end{equation} In particular, by using the scalar flat hypothesis, it follows that

\begin{equation}
\label{1cf}
\frac{area(\partial M)}{Vol(M)}\leq \sqrt{\frac{n^{2}}{2(n-1)f_{\max}}},
\end{equation} and this proves the assertion.

Next, if the equality holds in (\ref{1cf}), it suffices to use the equality case of Theorem  \ref{thmMainA} to conclude that $M^n$ must be isometric to a geodesic ball in $\Bbb{R}^n$ of  radius $\sqrt{2(n-1)f_{\max}}.$ So, the proof is completed. 

\end{proof}

\subsection{Proof of Corollary \ref{cor2thm1}}
\begin{proof} Since $\partial M$ is isometric to a standard sphere $\mathbb{S}^{n-1}(r)$ we may apply Theorem 11 in \cite{BLF} to deduce
$$H^2\leq\Big(\frac{n-1}{r}\Big)^2.$$ Therefore, using that $r>\sqrt{2(n-1)f_{\max}}$ we immediately conclude that $H^2<\frac{n-1}{2f_{\max}}$ and hence, the result follows by Theorem \ref{thmMainA}.
\end{proof}

\subsection{Proof of Theorem \ref{thmlocalA}}
First of all, we define a functional $\mathcal{F}:[0,\,f_{\max})\to \Bbb{R}$ by

\begin{equation}
\label{defU}
\mathcal{F}(t)= \displaystyle\left(\frac{R(f_{\max}^{2}-t^{2})+2n(f_{\max}-t)}{n(n-1)}\right)^{-\frac{n}{2}}\int_{\{f=t\}\cap E}|\nabla f|d\sigma.
\end{equation} 

We remark that $\mathcal{F}$ is well defined. Indeed, as it was previously mentioned $f$ is analytical and hence, it follows from \cite{Krantz} that the level sets of $f$ have locally finite $\mathcal{H}^{n-1}$-measure. Besides the level sets have finite hypersurface area. We also point out that $\mathcal{F}(t)$ is constant on the geodesic balls in space forms.

In order to prove Theorem \ref{thmlocalA} we need to provide a couple of lemmas. In the first one we will show that the function $\mathcal{F}$ is monotonically nonincreasing provided the mean curvature of $\Gamma_E$ is suitably bounded from below. This is the content of the following lemma.

\begin{lemma}
\label{lemAthmAlocal}
Let $(M^n,\,g,\,f)$ be an $n$-dimensional compact Miao-Tam critical metric with (possibly disconnected) boundary $\partial M,$ let $E$ be a single connected component of $M\setminus MAX(f).$ Suppose that $$H^{2}\ge \frac{n(n-1)}{Rf_{\max}^{2}+2nf_{\max}}\,\,\,\,\hbox{on}\,\,\,\Gamma_E.$$ Then $\mathcal{F}$ is monotonically nonincreasing.
\end{lemma}
\begin{proof}
Firstly, we consider the function
$$h=\frac{R(f_{\max}^{2}-f^{2})+2n(f_{\max}-f)}{n(n-1)}.$$ Then, one easily verifies that

\begin{eqnarray*}
{\rm div}(h^{-\frac{n}{2}}\nabla f)&=&h^{-\frac{n}{2}}\Delta f-\frac{n}{2}h^{-\frac{n}{2}-1}\langle\nabla f,\nabla h\rangle\\
&=&h^{-\frac{n}{2}-1}\Delta f(h-|\nabla f|^{2}).
\end{eqnarray*} Hence, it follows from Lemma \ref{lemK1} that

$${\rm div}\left[\Big(\frac{R(f_{\max}^{2}-f^{2})+2n(f_{\max}-f)}{n(n-1)}\Big)^{-\frac{n}{2}}\nabla f\right]\leq0.$$  Upon integrating this over $\{t_{1}\leq f\leq t_{2}\}\cap E,$ with $t_{1}<t_{2},$ we use the Stokes' formula to infer

\begin{eqnarray*}
0&\geq&\int_{\{t_{1}\leq f\leq t_{2}\}\cap E}{\rm div}(h^{-\frac{n}{2}}\nabla f)d\sigma=\int_{\partial\{t_{1}\leq f\leq t_{2}\}\cap E} h^{-\frac{n}{2}}\langle\nabla f, \nu\rangle dS\\
&=&\int_{\{f=t_{1}\}\cap E}h^{-\frac{n}{2}}\left\langle \nabla f,-\frac{\nabla f}{|\nabla f|}\right\rangle dS+\int_{\{f=t_{2}\}\cap E}h^{-\frac{n}{2}}\left\langle \nabla f,\frac{\nabla f}{|\nabla f|}\right\rangle dS\\
&=&-\int_{\{f=t_{1}\}\cap E}h^{-\frac{n}{2}}|\nabla f| dS+\int_{\{f=t_{2}\}\cap E}h^{-\frac{n}{2}}|\nabla f| dS\\
&=&-\mathcal{F}(t_{1})+\mathcal{F}(t_{2}),
\end{eqnarray*} where $\nu=\frac{\nabla f}{|\nabla f|}$ stands for the unit normal to $\{f=t_{1}\}$ and $\{f=t_{2}\},$ respectively. Therefore, it follows that $\mathcal{F}(t_{2})\leq \mathcal{F}(t_{1})$ for $t_{1}<t_{2}.$ This concludes the proof of the lemma.
\end{proof}

Proceeding, we need to analyse the behaviour of the functional $\mathcal{F}$ when $t\to f_{\max}.$ In this sense, we have the following lemma.

\begin{lemma}
\label{lemBthmAlocal}
Let $(M^n,\,g,\,f)$ be an $n$-dimensional compact Miao-Tam critical metric with (possibly disconnected) boundary $\partial M,$ let $E$ be a single connected component of $M\setminus MAX(f).$ Suppose that $\mathcal{H}^{n-1}(MAX(f)\cap \overline{E})>0.$ Then $\displaystyle \lim\limits_{t\rightarrow  f_{\max}^{-}}\mathcal{F}(t)=+\infty.$
\end{lemma}

\begin{proof}
First of all, we use the Łojasiewicz inequality (cf. \cite{Loja}, Theorem 4)  to deduce that for each point $p\in MAX(f)$ there exists a neighborhood $V_{p}\subset M$ of $p$ and real numbers $c_{p}>0$ and $0< \theta_{p}< 1,$ such that for each $x\in V_{p}$ we have $$|\nabla f|(x)\geq c_{p}\big(f_{\max}-f(x))^{\theta_{p}}.$$ Moreover, restricting the neighborhood $V_{p}$ if necessary, we may assume that $f_{\max}-f<1$ on $V_{p}$ such that for every $x\in V_{p}$ we have $$|\nabla f|(x)\geq c\big(f_{\max}-f(x)).$$ Next, taking into account that $MAX(f)$ is compact, it admits a finite covering given by $V_{p_1},\ldots,V_{p_k}.$ In particular, choosing $c=\min \{c_{p_1},\ldots, c_{p_k}\},$ we have that $V=V_{p_1}\cup \ldots\cup V_{p_k}$ is a neighborhood of $MAX(f)$ so that
\begin{equation}
\label{eqklp}
|\nabla f|(x)\geq c\big(f_{\max}-f(x)),
\end{equation} for all $x\in V.$ Again, since $M$ is compact and $f$ is analytical it holds that for $t$ sufficiently close to $f_{\max}$ we have $\{f=t\}\cap E\subset V.$ 

From now on we need to analyse the behaviour of $\mathcal{F}(t)$ for these values of $t.$ To do so, observe that from (\ref{eqklp}) it holds

\begin{eqnarray}
\label{auxINQFt}
\mathcal{F}(t)&=& \displaystyle\left(\frac{R(f_{\max}^{2}-t^{2})+2n(f_{\max}-t)}{n(n-1)}\right)^{-\frac{n}{2}}\int_{\{f=t\}\cap E}|\nabla f|d\sigma \nonumber\\&\geq& c \left[\frac{R(f_{\max}^{2}-t^{2})+2n(f_{\max}-t)}{n(n-1)}\right]^{-\frac{n}{2}}(f_{\max}-t)\cdot area(\{f=t\}\cap E)\nonumber\\&=&\frac{c}{\big(f_{\max}-t\big)^ {\frac{n-2}{2}}}\left[\frac{R(f_{\max}+t)+2n}{n(n-1)}\right]^ {-\frac{n}{2}}\cdot area(\{f=t\}\cap E).
\end{eqnarray} Therefore, to conclude the proof of the lemma we need only to show that if $\mathcal{H}^{n-1}(MAX(f)\cap\overline{E})>0,$ then $$\limsup_{t\rightarrow  f_{\max}^{-}}\, area(\{f=t\}\cap E)>0.$$ But, to prove this, it suffices to repeat exactly the same arguments used by Bor\-ghini and Mazzieri \cite{BM1} in the last part of the proof of Proposition 5.4. We highlight that the argument used by them in this part of the proof is in fact general. So, we omit the details, leaving them to the interested reader.  
\end{proof}

\subsubsection{Conclusion of the proof of Theorem  \ref{thmlocalA}}

\begin{proof}
We start the proof supposing that 

$$H^{2}\geq \frac{n(n-1)}{Rf_{\max}^{2}+2nf_{\max}}\,\,\,\,\hbox{on}\,\,\,\,\Gamma_E.$$ In this situation, it follows from Lemma \ref{lemAthmAlocal} that the function $\mathcal{F}(t)$ is monotonically nonincreasing. At the same time, we deduce $$\lim_{t\rightarrow  f_{\max}^{-}}\mathcal{F}(t)\leq \mathcal{F}(0),$$ and the value of $\mathcal{F}(0)$ yields

\begin{eqnarray*}
\lim_{t\rightarrow f_{\max}^{-}}\mathcal{F}(t)&\leq& \left(\frac{Rf^{2}_{\max}+2n f_{\max}}{n(n-1)}\right)^{-\frac{n}{2}}\int_{\Gamma_E}|\nabla f|\\
&\leq&\left(\frac{Rf^{2}_{\max}+2n f_{\max}}{n(n-1)}\right)^{-\frac{(n-1)}{2}}\cdot area(\Gamma_E),
\end{eqnarray*} where we have omitted the volume form. Therefore, we obtain $$\lim_{t\rightarrow  f_{\max}^{-}}\mathcal{F}(t)<\infty.$$ Now, we invoke Lemma \ref{lemBthmAlocal} to conclude that $\mathcal{H}^{n-1}(MAX(f)\cap \overline{E})=0.$ This imme\-diately guarantees that $MAX(f)\cap \overline{E}$ can not disconnect the domain $E$ from the rest of the manifold $M.$ Hence, $E$ is the only connnected component of $M\setminus MAX(f),$ which implies that $\partial M\cap E=\partial M.$ Finally, it suffices to apply Theorem  \ref{thmMainA} to conclude the proof of Theorem \ref{thmlocalA}. 
\end{proof}

\vspace{0.70cm}
We now pass to the proof of Theorem \ref{thmB}.

\subsection{Proof of Theorem \ref{thmB}}
\begin{proof} 
Firstly, we already know by Lemma \ref{Shen} that

\begin{equation}
\label{eq12a}
{\rm div}\left[\frac{1}{f}\nabla|\nabla f|^{2}-\frac{2\Delta f}{nf}\nabla f\right]=\frac{2}{f}|\mathring{Hessf}|^{2}\ge 0.
\end{equation} Next, taking into account that $M^n$ is compact and using the properties of the function $f,$ it follows that there exists a $\varepsilon>0$ such that the set $\{f=t\}$ is regular for every $0\le t\le \varepsilon$ and $f_{max}-\varepsilon\le t<f_{max}.$ 

By one hand, upon integrating (\ref{eq12a}) over $\{\varepsilon<f<f_{max}-\varepsilon\}\cap E$ we obtain

\begin{equation}
\label{9jk}
\int_{\{f=f_{max}-\varepsilon\}\cap E}\left\langle \frac{1}{f}\Big(\nabla |\nabla f|^{2}-\frac{2\Delta f}{n}\nabla f\Big),\,\nu\right\rangle\ge \int_{\{f=\varepsilon\}\cap E}\left\langle\frac{1}{f}\Big(\nabla |\nabla f|^{2}-\frac{2\Delta f}{n}\nabla f\Big),\,\nu\right\rangle,
\end{equation}  where $\nu=\frac{\nabla f}{|\nabla f|}$ is the unit normal to $\{\varepsilon<f<f_{max}-\varepsilon\}\cap E$ and we have omitted the volume forms.

On the other hand, we already know from (\ref{eq34dg}) that

\begin{equation*}
\frac{1}{f}\left\langle \nabla |\nabla f|^{2}-\frac{2\Delta f}{n}\nabla f,\,\frac{\nabla f}{|\nabla f|}\right\rangle=2|\nabla f|\Big( Ric(\nu,\nu)-\frac{R}{n}\Big).
\end{equation*} Plugging this data into (\ref{9jk}) yields

\begin{equation}
\label{tg1}
\int_{\{f=f_{max}-\varepsilon\}\cap E}|\nabla f|\Big(Ric(\nu,\,\nu)-\frac{R}{n}\Big)\geq \int_{\{f=\varepsilon\}\cap E}|\nabla f|\Big(Ric(\nu,\nu)-\frac{R}{n}\Big).
\end{equation}

Now, we claim that 
\begin{equation}
\label{oi9}
\liminf_{\varepsilon \to 0}\int_{\{f=f_{max}-\varepsilon\}\cap E}|\nabla f|\Big(Ric(\nu, \nu)-\frac{R}{n}\Big)=0.
\end{equation} To prove this, since $Ric(\nu, \nu)-\frac{R}{n}$ is bounded, it suffices to show that

$$\liminf_{t\to f_{max}}\int_{\{f=t\}\cap E}|\nabla f|=0.$$ Indeed, by Theorem \ref{reverse_Lojasiewicz} and Remark \ref{remA}, given a $\theta<1,$ there exists a neighborhood  $V$ of $MAX(f)$ and a real number $C>0$ such that
$$|\nabla f|^{2}(x)\leq C\left(f_{\max}-f(x)\right)^{\theta},$$ for every $x\in V.$ Observe that we are considering the same arguments used in the beginning of the proof of Lemma \ref{lemBthmAlocal} in order to obtain $V,$ that is, $t$ is chosen sufficiently close to $f_{\max}$ so that $\{f=t\}\cap E\subset V.$

Proceeding, for $t$ close to $f_{\max}$ we immediately have

$$\int_{\{f=t\}\cap E}|\nabla f|\leq C^{\frac{1}{2}}\left(f_{\max}-t\right)^{\frac{\theta}{2}}area\big(\{f=t\}\cap E\big).$$ 
From this, it follows that

$$\liminf_{t\to f_{max}}\int_{\{f=t\}\cap E}|\nabla f|=0,$$ as wished.

Now, taking $\liminf\limits_{\varepsilon \to 0}$ in (\ref{tg1}) we achieve at

\begin{equation*}
\int_{\Gamma_{E}}|\nabla f|\Big(Ric(\nu,\nu)-\frac{R}{n}\Big)\leq 0.
\end{equation*} At the same time, by Gauss Equation it holds

\begin{equation}
\label{2a}
2Ric(\nu,\nu)+R^{\Gamma_{E}}=R+\frac{(n-2)}{(n-1)}H^{2},
\end{equation} where $R^{\Gamma_{E}}$ stands for the scalar curvature of $(\Gamma_{E},g_{|_{\Gamma_{E}}}).$ Then, it follows that 

\begin{equation}
\int_{\Gamma_{E}}|\nabla f|\Big(R^{\Gamma_{E}}-\frac{n-2}{n}R-\frac{n-2}{n-1}H^{2}\Big)\geq 0,
\end{equation} as asserted.

Finally, if the equality holds it suffices to use (\ref{eq12a}) jointly with (\ref{IdRicHess}) to conclude that $|\mathring{Ric}|^{2}=0$ on $E,$ i.e., $E$ is an Einstein manifold and hence, by the analyticity of the metric we deduce that $M^{n}$ is Einstein. Therefore, we apply once more Theorem 1.1 in \cite{miaotamTAMS} to conclude that $M^{n}$ is isometric to a geodesic ball in $\Bbb{R}^n,$ $\Bbb{H}^{n}$ or $\Bbb{S}^n.$ So, the proof is finished. 
\end{proof}

\end{document}